\documentclass[a4paper]{amsart}

\usepackage{amsmath,amstext,amssymb,mathrsfs,amscd,amsthm,indentfirst}
\usepackage{amsfonts}
\usepackage{mathtools}

\usepackage{booktabs}
\usepackage{verbatim}
\usepackage{enumerate}

\usepackage{graphicx}

\numberwithin{equation}{section}

\newcommand{\bF}{\mathbb{F}}

\newcommand{\bH}{\mathbb{H}}

\newcommand{\bO}{\mathbb{O}}

\newcommand{\bR}{\mathbb{R}}

\newcommand{\bZ}{\mathbb{Z}}

\newcommand{\MTtheta}{\bold{MT \theta}}

\newcommand\lra{\longrightarrow}
\newcommand\lla{\longleftarrow}

\newcommand\Diff{\mathrm{Diff}}

\newcommand\hocolim{\operatorname*{hocolim}}

\newcommand\Ker{\operatorname*{Ker}}

\newcommand{\fr}{\mathrm{fr}}

\newcommand{\Fr}{\mathrm{Fr}}
\newcommand{\R}{\bR}

\newcommand{\Hom}{\mathrm{Hom}}

\newcommand{\Aut}{\mathrm{Aut}}

\newcommand{\IM}{\mathrm{Im}}

\renewcommand{\epsilon}{\varepsilon}

\newcommand{\Spin}{\mathrm{Spin}}

\mathchardef\ordinarycolon\mathcode`\:
\mathcode`\:=\string"8000
\begingroup \catcode`\:=\active
  \gdef:{\mathrel{\mathop\ordinarycolon}}
\endgroup

\usepackage[all]{xy}
\CompileMatrices

\usepackage{amsthm}
\theoremstyle{plain}

\newtheorem{theorem}{Theorem}[section]

\newtheorem{proposition}[theorem]{Proposition}
\newtheorem{lemma}[theorem]{Lemma}

\newtheorem{corollary}[theorem]{Corollary}

\theoremstyle{definition}
\newtheorem{definition}[theorem]{Definition}

\newtheorem{construction}[theorem]{Construction}

\theoremstyle{remark}

\newtheorem{remark}[theorem]{Remark}
\newtheorem*{remark*}{Remark}

\SelectTips{eu}{10}

\title[Abelian quotients of mapping class groups]{Abelian quotients of mapping class groups of highly connected manifolds}

\author{S{\o}ren Galatius}
\email{galatius@stanford.edu}
\address{Department of Mathematics\\
	Stanford University\\
	Stanford CA, 94305}

\author{Oscar Randal-Williams}
\email{o.randal-williams@dpmms.cam.ac.uk}
\address{Centre for Mathematical Sciences\\
Wilberforce Road\\
Cambridge CB3 0WB\\
UK}

\begin{document}
\begin{abstract}
  We compute the abelianisations of the mapping class groups of the
  manifolds $W_g^{2n} = g(S^n \times S^n)$ for $n \geq 3$ and $g \geq
  5$.  The answer is a direct sum of two parts.  The first part arises
  from the action of the mapping class group on the middle homology,
  and takes values in the abelianisation of the automorphism group of
  the middle homology.  The second part arises from bordism classes of
  mapping cylinders and takes values in the quotient of the stable
  homotopy groups of spheres by a certain subgroup which in many cases
  agrees with the image of the stable $J$-homomorphism.  We relate its
  calculation to a purely homotopy theoretic problem.
\end{abstract}
\maketitle

\section{Introduction}

Let $W_g^{2n} = g(S^n \times S^n)$ denote the $g$-fold connected sum
and choose a fixed closed disc $D^{2n} \subset W_g^{2n}$. Let
$\Diff^+(W_g^{2n})$ be the topological group of orientation preserving
diffeomorphisms of $W_g^{2n}$, and $\Diff(W_g^{2n}, D^{2n})$ be the
subgroup of those diffeomorphisms which fix an open neighbourhood of
the disc. Define the \emph{mapping class groups}
$$\Gamma_{g,1}^n = \pi_0(\Diff(W_g^{2n}, D^{2n})) \quad\quad\quad \Gamma_{g}^n = \pi_0(\Diff^+(W_g^{2n})).$$
There is a homomorphism $\gamma : \Gamma_{g,1}^n \to \Gamma_g^n$,
which simply forgets that diffeomorphisms fix a disc. We will
construct two abelian quotients of these groups, one coming from
arithmetic properties of the intersection form of $W_g^{2n}$, and one
coming from a cobordism theoretic construction. Together, these will
give the abelianisation of either group.

\begin{construction}
  Recall that Wall \cite{Wall62} has constructed for each
  $(n-1)$-connected $2n$-manifold $W$ a certain quadratic form $Q_W$,
  which we shall describe later, whose underlying bilinear form is the intersection form on $H_n(W;\bZ)$. Diffeomorphisms of the manifold act by automorphisms
  of this quadratic form, so there is a group homomorphism
  \begin{equation*}
    \hat{f} : \Gamma_{g}^n \lra \mathrm{Aut}(Q_{W_g^{2n}}),
  \end{equation*}
  from which we can construct the map $f : \Gamma_{g}^n \to
  H_1(\mathrm{Aut}(Q_{W_g^{2n}}))$ to an abelian group. We will also
  write $\hat{f}$ for the composition $\hat{f} \circ \gamma :
  \Gamma_{g,1}^n \to \Gamma_g^n \to \mathrm{Aut}(Q_{W_g^{2n}})$, and
  similarly with $f$.
\end{construction}

\begin{construction}
  Let $\varphi \in \Diff(W_g^{2n}, D^{2n})$ be a diffeomorphism of
  $W_{g}^{2n}$ which is the identity on a fixed disc $D^{2n} \subset
  W_g^{2n}$. We may form the mapping torus
  \begin{equation*}
    T_\varphi = W_g^{2n}
  \times [0,1] / (x,0) \sim (\varphi(x),1),
  \end{equation*}
  which is a $(2n+1)$-dimensional manifold fibering over $S^1$, and
  contains an embedded $D^{2n} \times S^1$ given by the disc fixed by
  $\varphi$.  The $(n-1)$-connected manifold obtained by surgery along
  this embedded $D^{2n} \times S^1$ shall be denoted $T_\phi'$.  This
  construction is often called an open book. By obstruction theory, a
  map $\tau: T_\phi' \to BO$ classifying its stable normal bundle
  admits a lift $\ell: T_\phi' \to BO\langle n\rangle$, unique up to
  homotopy, where $BO\langle n \rangle \to BO$ denotes the
  $n$-connected cover. The pair $(T'_\varphi, \ell)$ represents an
  element of $\Omega_{2n+1}^{\langle n \rangle}$, the cobordism theory
  associated to the map $BO\langle n\rangle \to BO$, and one easily
  verifies that the function
  \begin{align*}
    t : \Gamma_{g,1}^n & \lra \Omega_{2n+1}^{\langle n \rangle}\\
    \varphi & \longmapsto [T_\varphi', \ell]
  \end{align*}
  is a group homomorphism.
\end{construction}

Our main theorem, proved in Sections
\ref{sec:low-dimens-cases}--\ref{sec:refin-mapp-torus} below, is that
these two homomorphisms combine to give the maximal abelian quotient
of the group $\Gamma_{g,1}^n$.

\begin{theorem}\label{thm:Main}
For all $n$ and $g$ (except we require $g \geq 2$ if $n=2$) the map
$$t \oplus f : \Gamma_{g,1}^n \lra \Omega_{2n+1}^{\langle n \rangle} \oplus H_1(\mathrm{Aut}(Q_{W_g^{2n}}))$$
is surjective, and for $n \neq 2$ and $g \geq 5$ it is the abelianisation. Furthermore, in this range
$$H_1(\mathrm{Aut}(Q_{W_g^{2n}})) \cong \begin{cases}
(\bZ/2)^2 & \text{$n$ even}\\
0 & \text{$n=1$, $3$ or $7$}\\
\bZ/4 & \text{otherwise}.
\end{cases}$$
\end{theorem}

We obtain the following table describing $H_1(\Gamma_{g,1}^n;\bZ)$ for small $n$, using known calculations of $\Omega_*^{\langle 3 \rangle} = \Omega_*^{\mathrm{Spin}}$ and $\Omega_*^{\langle 7 \rangle} = \Omega_*^{\mathrm{String}}$ (see \cite[p.\ 201]{MilnorSpin} for the former and \cite{Giambalvo} for the latter).
\begin{table}[h]
\centering
\caption{Abelianisations of $\Gamma_{g,1}^n$ for $g \geq 5$.}
\label{table:1}
\begin{tabular}{c|c c c c c c c}
$n$              & 1 & 2 & 3 & 4 & 5  & 6 & 7 \\
\hline
$H_1(\Gamma_{g,1}^n;\bZ)$ & 0 & $(\bZ/2)^2 \oplus ?$ & 0 & $(\bZ/2)^4$ & $\bZ/4$ & $(\bZ/2)^2 \oplus \bZ/3$ & $\bZ/2$  \\
\end{tabular}
\end{table}

In Section \ref{sec:Kreck} we compare our work with that of Kreck
\cite{KreckAut}, who has described the groups $\Gamma_{g,1}^n$ up to
extension problems. Using Theorem \ref{thm:Main} we are able to
resolve these extension problems when $n=6$ or $n \equiv 5 \mod 8$,
and hence give a complete description of these mapping class groups.

\subsection{The cobordism groups $\Omega_{2n+1}^{\langle n \rangle}$}

In light of Theorem \ref{thm:Main}, it is of interest to describe the cobordism group $\Omega_{2n+1}^{\langle n \rangle}$ in terms of more familiar objects. There is a homomorphism
$$\rho : \Omega_{2n+1}^{\fr} \lra \Omega_{2n+1}^{\langle n \rangle}$$
from framed cobordism obtained by simply remembering that a stably tangentially framed manifold in particular has a $BO\langle n \rangle$-structure. The cobordism theoretic interpretation of the $J$-homomorphism
$$J : \pi_{2n+1}(SO) \lra \pi_{2n+1}^s = \Omega_{2n+1}^{\fr}$$
is that it sends a map $f: S^{2n+1} \to SO$ to the stably framed
manifold obtained by taking the $(2n+1)$-sphere with its
usual---bounding---stable framing, and changing the framing using
$f$. The resulting stable framing need not extend over $D^{2n+2}$, but
the $BO\langle n \rangle$-structure does always extend (as the map
$BO\langle n\rangle \to BO$ is $n$-co-connected), so $\rho \circ
J$ is trivial. Thus there is an induced map
$$\rho' : \mathrm{Coker}(J)_{2n+1} \lra \Omega_{2n+1}^{\langle n \rangle}.$$
It follows from work of Stolz that this map is an isomorphism in many cases.

\begin{theorem}[Stolz]\label{thm:RhoIso}
The map $\rho'$ is surjective, and is an isomorphism if either
\begin{enumerate}[(i)]
\item $n+1 \equiv 2 \mod 8$ and $n+1 \geq 18$,

\item $n+1 \equiv 1 \mod 8$ and $n+1 \geq 113$,

\item $n+1 \not\equiv 0,1,2,4 \mod 8$.
\end{enumerate}
\end{theorem}

In the cases not covered by this theorem, the kernel of $\rho'$ is at
most $\bZ/2$ if $n+1 \equiv 1,2 \mod 8$, and cyclic if $n+1 \equiv 0
\mod 4$. We give more detailed information in Section \ref{sec:Filtr}.

\subsection{Closed manifolds}

We can also use Theorem \ref{thm:Main} to calculate the abelianisation
of the mapping class group $\Gamma_g^n$, of orientation preserving
diffeomorphisms of the closed manifolds $W_g^{2n}$, because of the
following result of Kreck.

\begin{lemma}[Kreck]\label{lem:Closed}
  The map $\gamma: \Gamma_{g,1}^n \to \Gamma_{g}^n$ is an isomorphism
  for $n \geq 3$.
\end{lemma}
\begin{proof}
  The homotopy fiber sequences
  \begin{equation*}
    \Fr^+(W_g^{2n}) \lra B\Diff(W_g^{2n}, D^{2n}) \lra B\Diff^+(W_g^{2n})
  \end{equation*}
  and
  \begin{equation*}
    SO(2n) \lra \Fr^+(W_g^{2n}) \lra W_g^{2n}
  \end{equation*}
  induce long exact sequence in homotopy groups, from which it is easy
  to see that $\Gamma_{g,1}^n \to \Gamma_g^n$ is surjective with
  kernel either trivial or $\bZ/2$, as long as $n \geq 3$.

  Kreck proves that the kernel is in fact trivial: Combine the
  discussion at the bottom of page 657 of \cite{KreckAut} with the
  fact that the manifolds $W_g^{2n}$ bound the parallelisable
  manifolds $\natural^g S^n \times D^{n+1}$, so the element
  $\Sigma_{W_g^{2n}}$ is trivial by Lemma 3b of that paper.
\end{proof}

\subsection{Perfection}

Recall that a group is called \emph{perfect} if it is equal to its
derived subgroup, or equivalently if its abelianisation is
trivial. Table \ref{table:1} shows that $\Gamma_{g,1}^n$ (or $\Gamma_g^n$, by
Lemma \ref{lem:Closed}) is perfect for $n=1$ or
$n=3$ and $g \geq 5$, but the fact that
$H_1(\mathrm{Aut}(Q_{W_g^{2n}}))$ is trivial only for $n=1,3,7$ and
the fact that $\Omega_{15}^{\langle 7 \rangle} \neq 0$ means that
these are the only examples.

\begin{corollary}\label{cor:perfection}
  For $g \geq 5$, the groups $\Gamma_{g,1}^n$ and $\Gamma_{g}^n$ are
  perfect if and only if $n$ is 1 or 3.
\end{corollary}

If we denote by $\mathring{W}^{2n}_{g,1}$ the complement of the chosen
disc $D^{2n}$ in $W_{g}^{2n}$, then the group $\Diff(W_g^{2n},
D^{2n})$ is isomorphic to $\Diff_c(\mathring{W}^{2n}_{g,1})$, the
group of compactly supported diffeomorphisms of
$\mathring{W}^{2n}_{g,1}$. Thurston \cite{Thurston} has proved that
for any manifold $M$ without boundary the identity component
$\Diff_c(M)_0^\delta$, considered as a discrete group, is perfect (in
fact, it is simple). Thus the extension of discrete groups
$$1 \lra \Diff_c(\mathring{W}^{2n}_{g,1})^\delta_0 \lra
\Diff(W_g^{2n}, D^{2n})^\delta \lra \Gamma_{g,1}^n \lra 1$$ shows that
the discrete group $\Diff(W_g^{2n}, D^{2n})^\delta$ is perfect if and
only $\Gamma_{g,1}^n$ is, and more generally that the abelianisation
of the discrete group $\Diff(W_g^{2n}, D^{2n})^\delta$ is also
described by Theorem \ref{thm:Main}. Similarly, the abelianisation of
the discrete group $\Diff^+(W_g^{2n})^\delta$ is also described by
Theorem \ref{thm:Main}. See \cite{Nariman} for more
information about the homology of $\Diff(W_g^{2n}, D^{2n})^\delta$.

\subsection{Acknowledgements}

S.\ Galatius was partially supported by National Science Foundation
grants DMS-1105058 and DMS-1405001, O.\ Randal-Williams was supported
by the Herchel Smith Fund, and both authors were supported by ERC
Advanced Grant No.\ 228082 and the Danish National Research Foundation
through the Centre for Symmetry and Deformation.


\section{Wall's quadratic form}

The fibration $S^n \to BO(n) \to BO(n+1)$ gives a long exact sequence
on homotopy groups
\begin{equation*}
  \cdots \lra \pi_{n+1}(BO(n+1)) \overset{\partial}\lra \bZ \overset{\tau}\lra \pi_n(BO(n)) \overset{s}\lra \pi_n(BO(n+1)) \lra 0,
\end{equation*}
and we let $\Lambda_n = \IM(\partial) \subset \bZ$. The map $\tau$
sends 1 to the map which classifies the tangent bundle of the
$n$-sphere, so $\Lambda_n$ is trivial if $n$ is even, $\bZ$ if $n=1$,
$3$ or $7$, and $2\bZ$ otherwise, by the Hopf invariant 1 theorem. The
data $((-1)^n, \Lambda_n)$ is a \emph{form parameter} in the sense of
Bak \cite{Bak, Bak2}.

Suppose that $n \geq 4$, and let $W$ be an $(n-1)$-connected
$2n$-manifold which is stably parallelisable. We will describe how to
associate to it a non-degenerate quadratic form $Q_W$ having form
parameter $((-1)^n, \Lambda_n)$, following Wall \cite{Wall62}. The
$\bZ$-module
$$\pi_n(W) \cong H_n(W;\bZ)$$
has a $(-1)^n$-symmetric bilinear form
$$\lambda : H_n(W;\bZ) \otimes H_n(W;\bZ) \lra \bZ$$
given by the intersection form, which is non-degenerate by
Poincar{\'e} duality. If $x=[f] \in \pi_n(W)$, then by a theorem of
Haefliger \cite{Haefliger} as $n \geq 4$ we may represent it uniquely
up to isotopy by an embedding $f : S^n \hookrightarrow W$, which has
an $n$-dimensional normal bundle which is stable trivial. This
represents an element
\begin{equation*}
  \alpha(x) \in \bZ/\Lambda_n = \Ker\left(\pi_n(BO(n)) \overset{s}\to
    \pi_n(BO(n+1)\right),
\end{equation*}
and Wall has shown that this satisfies
\begin{enumerate}[(i)]
\item $\alpha(a \cdot x) = a^2 \cdot \alpha(x)$, for $a \in \bZ$,
\item $\alpha(x+y) = \alpha(x)+\alpha(y) + \lambda(x,y)$, where
  $\lambda(x,y)$ is reduced modulo $\Lambda_n$.
\end{enumerate}
Thus the data $(\pi_n(W), \lambda, \alpha)$ is a quadratic form with
form parameter $((-1)^n, \Lambda_n)$.

\begin{remark}
  This construction above does not quite work for $n \leq 3$, as
  Haefliger's theorem does not apply, but we can proceed anyway.  When
  $n=1$ or $3$ we have $\bZ/\Lambda_n = \{0\}$ and so a quadratic form
  with parameter $(-1, \Lambda_n)$ should be a module with
  skew-symmetric bilinear form. We take $H_n(W_g;\bZ)$ with its
  intersection form.

  When $n=2$ we have $\bZ/\Lambda_2 = \bZ$ and so a quadratic form
  with parameter $(1, \Lambda_2)$ should be an even symmetric bilinear
  form. The intersection form on $H_2(W_g;\bZ)$ \emph{is} even, so we
  can take this.
\end{remark}

By construction, it is clear that if $\varphi : W_0 \to W_1$ is a
diffeomorphism then $\varphi_* : H_n(W_0;\bZ) \to H_n(W_1;\bZ)$ is a
morphism of quadratic forms. The most elementary quadratic form is the
\emph{hyperbolic form}
\begin{equation*}
  H = \left(\bZ^2 \text{ with basis } e, f ; 
    \begin{pmatrix}
      0 & 1   \\
      (-1)^n & 0
    \end{pmatrix}
    ; \alpha(e)=\alpha(f)=0\right).
\end{equation*}
The manifold $W_g^{2n} = g(S^n \times S^n)$ has associated quadratic form $H^{\oplus g}$, the direct sum
of $g$ copies of the hyperbolic form, and so we have a homomorphism
$$\hat{f} : \Gamma_g^n \lra \mathrm{Aut}(H^{\oplus g}).$$
Kreck \cite{KreckAut} has shown that this map is surjective for $n
\geq 3$, Wall \cite{Wall4Mfld} has shown it is surjective for $n=2$ as
long as $g \geq 5$, and it is well-known to be surjective for $n=1$
and all $g$. We obtain an abelian quotient
\begin{equation}
  f : \Gamma_g^n \lra H_1(\mathrm{Aut}(H^{\oplus g});\bZ).
\end{equation}

\begin{proposition}\label{prop:AutAb}
  There are isomorphisms
  \begin{equation*}
    H_1(\mathrm{Aut}(H^{\oplus g});\bZ) \cong 
    \begin{cases}
      (\bZ/2)^2 & \text{$n$ even}\\
      0 & \text{$n=1$, $3$ or $7$}\\
      \bZ/4 & \text{otherwise}.
    \end{cases}
  \end{equation*}
  as long as $g \geq 5$.
\end{proposition}
\begin{proof}
  By Charney's stability theorem \cite{Charney} for the homology of
  automorphism groups of quadratic forms over a PID, the group
  $H_1(\mathrm{Aut}(H^{\oplus g});\bZ)$ is independent of $g$ as long
  as $g \geq 5$. In fact, the statement in \cite{Charney} claims this
  only for $g \geq 6$, but using the slightly improved connectivity
  for the necessary poset / simplicial complex which is established in
  \cite[Theorem 3.2]{GR-WStability} this can be improved to $g \geq 5$
  (the poset $HU_g = HU(H^{\oplus g})$ of \cite{Charney} is the face
  poset of the simplicial complex $K^a(H^{\oplus g})$ of
  \cite{GR-WStability}, so they have homeomorphic geometric
  realisations).

  If $n$ is even, then $\mathrm{Aut}(H^{\oplus g}) = \mathrm{O}_{g,g}(\bZ)$ is
  the indefinite orthogonal group over the integers. This is a
  subgroup of $\mathrm{O}_{g,g}(\bR)$, which has maximal compact subgroup
  $\mathrm{O}_g(\bR) \times \mathrm{O}_g(\bR)$; the determinants of these two factors
  provides a surjective homomorphism $a: \mathrm{O}_{g,g}(\bZ) \to
  (\bZ/2)^2$. In \cite[Theorem 1.7]{GHS} it is shown that a certain
  index 4 normal subgroup $\widetilde{\mathrm{SO}}^+_{g,g}(\bZ)$ of
  $\mathrm{O}_{g,g}(\bZ)$, for the definition of which we refer to that paper,
  has trivial abelianisation. Thus the homomorphism $a$ is the
  abelianisation.

  If $n=1$, $3$ or $7$ then a quadratic form with parameter $(-1,
  \Lambda_n)$ is nothing but an antisymmetric bilinear form, so
  $\mathrm{Aut}(H^{\oplus g}) = \mathrm{Sp}_{2g}(\bZ)$ is the
  symplectic group over the integers. This is well-known to have
  trivial abelianisation, as long as $g \geq 3$.

  For the remaining odd $n$, $\mathrm{Aut}(H^{\oplus g}) = \mathrm{Sp}_{2g}^q(\bZ) \subset \mathrm{Sp}_{2g}(\bZ)$ is the subgroup of those symplectic matrices
  which stabilise the quadratic form $\alpha(e_i) =
  \alpha(f_i)=0$. The abelianisation of this group has been computed
  in \cite[Theorem 1.1]{JohnsonMillson} to be $\bZ/4$ as long as $g
  \geq 3$.
\end{proof}

\begin{remark}
  The argument above can be used to strengthen the ``only if'' part of
  Corollary~\ref{cor:perfection}: for $n \neq 1,3$, the mapping class
  groups are not perfect for \emph{any} $g \geq 1$.

  For $n=7$ this is in fact the case for $g \geq 0$, as the generator
  of $\Omega_{15}^{\langle 7 \rangle} = \bZ/2$ can be hit by a
  diffeomorphism supported inside a disc.  For $n \neq 7$ we argue as
  follows.  The matrix $\left (
    \begin{smallmatrix}
      -1 & 0 \\
      0 & -1
    \end{smallmatrix} \right )$ defines an element of
  $\mathrm{Aut}(H)$ for all $n$, and is easily seen to be realised by
  an element of $\Gamma_{1,1}^n$.  For $n$ even this maps to a
  non-trivial element of $(\bZ/2)^2$, and for $n$ odd apart from
  $1,3,7$, it follows from the formula \cite[p.\ 147]{JohnsonMillson}
  that it maps to the order-two element of $\bZ/4$.
\end{remark}

\section{Low-dimensional cases}
\label{sec:low-dimens-cases}
The cases $n < 3$ of Theorem \ref{thm:Main} require special treatment,
so let us dispense with them first. We will then focus on the generic
case $n \geq 3$.

\subsection{$\mathbf{n=1}$} The relevant bordism group is
$\Omega_3^{\langle 1 \rangle}$, third oriented bordism, which is
well-known to be zero. Thus the first part of Theorem \ref{thm:Main}
states that $\Gamma_{g,1}^1$ surjects onto the trivial group, which is
certainly true, and the second part states that the abelianisation of
$\Gamma_{g,1}^1$ is zero as long as $g \geq 5$. This is \cite[Theorem
1]{Powell} (which in fact only requires $g \geq 3$).

\subsection{$\mathbf{n=2}$} The relevant bordism group is
$\Omega_5^{\langle 2 \rangle}$, fifth Spin bordism, which is zero by
the results of \cite[p.\ 201]{MilnorSpin}. Thus in this case Theorem \ref{thm:Main}
just says that the map $f : \Gamma_{g,1}^2 \to
H_1(\mathrm{Aut}(Q_{W_g^4}))$ is surjective for $g \geq 2$. But the
homomorphism $\hat{f} : \Gamma_{g,1}^2 \to \mathrm{Aut}(Q_{W_g^4})$ is
already surjective in this case, by \cite[Theorem 2]{Wall4Mfld}. Though we do not require it for our results, Kreck \cite[Theorem 1]{KreckAut} has shown that for $g \geq 2$ the kernel of the surjective map $\Gamma_g^2 \to \mathrm{Aut}(Q_{W_g^4})$ is precisely the subgroup of those diffeomorphisms pseudoisotopic to the identity.


\section{Nontriviality of the mapping torus construction}\label{sec:Nontriv}
\label{sec:nontr-mapp-torus}

Using the stabilisation maps we have a homomorphism
$$d: \Gamma_{0,1}^n \lra \Gamma_{g,1}^n$$
for any $g$, and the group $\Gamma_{0,1}^n$---the mapping class group
of the sphere relative to a disc---is isomorphic to the group
$\Theta_{2n+1}$ of exotic $(2n+1)$-spheres via the clutching
construction.

\begin{lemma}\label{lem:Theta}\mbox{}
\begin{enumerate}[(i)]
\item The image of $\Theta_{2n+1}$ in $\Gamma_{g,1}^n$ is central. 

\item\label{it:Theta:2} The composition $\Theta_{2n+1} \overset{d}\to \Gamma_{g,1}^n \overset{t}\to \Omega_{2n+1}^{\langle n \rangle}$ is surjective, so in particular $t$ is surjective.

\item The composition $\Theta_{2n+1} \overset{d}\to \Gamma_{g,1}^n \overset{\hat{f}}\to \mathrm{Aut}(H^{\oplus g})$ is trivial.
\end{enumerate}
\end{lemma}
\begin{proof}
  Let $f$ be a diffeomorphism of $W_{g}$ fixing a neighbourhood $U$ of
  $D^{2n}$, and $g$ be a diffeomorphism supported in a disc disjoint
  from the marked one. Then $g$ is isotopic to a diffeomorphism $g'$
  supported in $U$ but still disjoint from $D^{2n}$, and now $g'$
  commutes with $f$. Thus $\IM(d) \subset \Gamma_{g,1}^n$ is central.

  The map $t \circ d$ sends an exotic $(2n+1)$-sphere to its
  $BO\langle n \rangle$-bordism class (such an exotic sphere has a
  canonical $BO\langle n \rangle$-structure by virtue of being
  highly-connected), so we must show that any $(2n+1)$-dimensional
  manifold with $BO\langle n \rangle$-structure $(W, \ell_W)$ is
  cobordant to a $n$-connected manifold (as it is then $2n$-connected
  by Poincar{\'e} duality).

  This follows from the methods of Kervaire and Milnor, specifically
  \cite[Theorem 6.6]{KervaireMilnor}. They work with manifolds which
  are stably parallelisable, but this is only used in two ways: to
  show that homotopy classes of dimension $* \leq n$ can be
  represented by framed embeddings, and to show that the trace of the
  surgery is stably parallelisable. A $BO\langle n \rangle$-structure
  still allows one to represent homotopy classes of dimension $* \leq
  n$ by framed embeddings, and a $BO\langle n \rangle$-structure can
  be induced on the trace of the surgery, too.

  Finally, a mapping class in the image of $d$ is supported in a small
  disc, and so acts trivially on the homology of $W_g$, so $\hat{f}
  \circ d$ is trivial.
\end{proof}

This lemma has the following implication regarding the kernel of the
mapping torus construction $t$.

\begin{corollary}\label{cor:Kernel}
  The kernel of the homomorphism $t : H_1(\Gamma_{g,1}^n) \to
  \Omega_{2n+1}^{\langle n \rangle}$ has cardinality at least 4 if $n
  \neq 1$, $3$ or $7$ and $g \geq 5$.
\end{corollary}
\begin{proof}
  Consider the commutative diagram
  \begin{equation*}
    \xymatrix{
      & \Theta_{2n+1} \ar[d]^-{d} \ar@{->>}[rd]\\
      \Ker(t) \ar@{-->}[rd] \ar@{^(->}[r] & H_1(\Gamma_{g,1}^n) \ar@{->>}[r]^-{t} \ar@{->>}[d] \ar@{->>}[rd]^-{f} & \Omega_{2n+1}^{\langle n \rangle} \\
      & H_1(\Gamma_{g,1}^n)/\Theta_{2n+1} \ar@{->>}[r] & H_1(\mathrm{Aut}(H^{\oplus g}))
    }
  \end{equation*}
  where the middle row is exact. Diagram chasing shows that the dashed
  arrow is surjective, and so $\Ker(t) \to H_1(\Gamma_{g,1}^n) \to
  H_1(\mathrm{Aut}(H^{\oplus g}))$ is surjective. The target has
  cardinality 4 in these cases by Proposition \ref{prop:AutAb}.
\end{proof}

\section{A refinement of the mapping torus construction}
\label{sec:refin-mapp-torus}

From now on we suppose that $n \geq 3$. The proof of the remainder of
Theorem \ref{thm:Main} uses two more involved theorems proved recently
by the authors, which concern not the mapping class groups but the
entire diffeomorphism groups of the manifolds $W_g^{2n}$. There
are continuous homomorphisms
$$\Diff(W_g^{2n}, D^{2n}) \lra \Diff(W_{g+1}^{2n}, D^{2n})$$
given by connect-sum with $W_1^{2n}$ inside the marked disc, and
extending diffeomorphisms by the identity. In \cite[Theorem
1.2]{GR-WStability} we showed that for $n \geq 3$ the maps on classifying
spaces
$$B\Diff(W_g^{2n}, D^{2n}) \lra B\Diff(W_{g+1}^{2n}, D^{2n})$$
induce homology isomorphisms in degrees $2\ast \leq g-3$. In particular, as long as $g \geq 5$ they induce
isomorphisms on first homology. The map $H_1(B\Diff(W_g^{2n},
D^{2n});\bZ) \to H_1(\Gamma_{g,1}^n;\bZ)$ is also an isomorphism,
which shows that the stabilisation map
$$H_1(\Gamma_{g,1}^n;\bZ) \lra H_1(\Gamma_{g+1,1}^n;\bZ)$$
is an isomorphism for $g \geq 5$.

Secondly, we showed how to identify the stable homology, that is, the
homology of $\hocolim_{g \to \infty} B\Diff(W_g^{2n}, D^{2n})$, as
follows.  Let $\theta_n : BO(2n)\langle n \rangle \to BO(2n)$ denote
the $n$-connected cover, and $\theta_n^*\gamma_{2n}$ denote the
pullback of the tautological $2n$-dimensional vector bundle. Write
$\MTtheta_n$ for the Thom spectrum of the virtual bundle
$-\theta_n^*\gamma_{2n} \to BO(2n)\langle n \rangle$. Parametrised
Pontrjagin--Thom theory provides maps
$$\alpha_g : B\Diff(W_g^{2n}, D^{2n}) \lra \Omega^\infty_0 \MTtheta_n$$
which assemble to a map $\alpha_\infty :
\hocolim_{g\to\infty}B\Diff(W_g^{2n}, D^{2n}) \to \Omega^\infty_0
\MTtheta_n$ which we show in \cite[Theorem 1.1]{GR-W2} induces an
isomorphism on homology as long as $n \geq 3$. Given these two
theorems, we are reduced to calculating $H_1(\Omega^\infty_0
\MTtheta_n)$.

Recall that $\MTtheta_n = \mathbf{Th}(-\theta_n^*\gamma_{2n} \to
BO(2n)\langle n \rangle)$. Let us write $\mathbf{MO}\langle n \rangle$
for the Thom spectrum\footnote{Some authors denote the $(n-1)$-connected cover of a space $X$ by $X\langle n \rangle$, and so write $\mathbf{MO}\langle n \rangle$ for the Thom spectrum associated to the $(n-1)$-connected cover of $BO$. We emphasise that our notation is different.} of the tautological bundle over $BO\langle n
\rangle$, so the stabilisation map induces a spectrum map
$$s : \MTtheta_n \lra \Sigma^{-2n} \mathbf{MO}\langle n \rangle.$$

\begin{lemma}
  The composition
  $$H_1(\Gamma_{g,1}^n) \overset{\sim}\longleftarrow H_1(B\Diff(W_g^{2n}, D^{2n})) \overset{\alpha_g}\lra H_1(\Omega^\infty_0\MTtheta_n) \cong \pi_1(\MTtheta_n) \overset{s_*}\lra \Omega_{2n+1}^{\langle n \rangle}$$
  agrees with the mapping torus construction $t$.
\end{lemma}
\begin{proof}
  Both apply the Pontryagin--Thom construction to the mapping
  torus.
\end{proof}

\subsection{A long exact sequence in stable homotopy}

Let us write $\mathbf{F}_n$ for the homotopy fibre of the
spectrum map $s$.  There is a commutative diagram
\begin{equation}\label{eq:1}
  \begin{aligned}
  \xymatrix{
    SO/SO(2n) \ar[r] \ar[d] &
    BO(2n)\langle n \rangle \ar[r] \ar[d]& BO(2n) \ar[d]\\
    \ast \ar[r] &
    BO\langle n \rangle \ar[r]& BO,
  }
  \end{aligned}
\end{equation}
where both squares are homotopy pullback.  The left square induces a
map of homotopy cofibres $\Sigma(SO/SO(2n)) \to (BO\langle
n\rangle)/(BO(2n)\langle n\rangle)$ which we see from the Serre
spectral sequence to be
$(3n+1)$-connected.  The whole diagram maps to $BO$, and may be
Thomified.  The map of cofibres, desuspended $(2n+1)$ times, gives a
map
\begin{equation*}
  \Sigma^{-2n} SO/SO(2n) \lra \mathbf{F}_n,
\end{equation*}
which is $n$-connected.

We may therefore rewrite the long exact sequence in stable homotopy
for the map $s:\MTtheta_n \to \Sigma^{-2n} \mathbf{MO}\langle n\rangle$ in the
following way,
\begin{equation}\label{eq:LES}
\begin{gathered}
  \xymatrix{
    & \cdots \ar[r]^-{s_*}& \pi_{2n+2}(\mathbf{MO}\langle n \rangle) \ar[lld]_{\partial_*}\\
    \pi_{2n+1}^s(SO / SO(2n)) \ar[r]& \pi_1(\MTtheta_n) \ar[r]^-{s_*} & \pi_{2n+1}(\mathbf{MO}\langle n \rangle) \ar[lld]\\
    \pi_{2n}^s(SO / SO(2n)) \ar[r] & \pi_0(\MTtheta_n) \ar[r]^-{s_*} & \pi_{2n}(\mathbf{MO}\langle n \rangle) \ar[r] & 0,
  }
\end{gathered}
\end{equation}
as $SO/SO(2n)$ is $(2n-1)$-connected, so $\pi_{2n-1}^s(SO/SO(2n))=0$. By this connectivity property, the group $\pi_{2n+1}^s(SO/SO(2n))$ is in the range of the Freudenthal suspension theorem as long as $2n+1 \leq 2(2n-1)$ i.e.\ $n \geq 2$, as is $\pi_{2n}^s(SO/SO(2n))$. Thus to compute these groups we may as well compute the associated unstable homotopy groups of $SO/SO(2n)$. As $SO(2n+m) \to SO$ is $(2n+m-1)$-connencted, the homotopy groups of $SO/SO(2n)$ agree with those of the Stiefel manifold $V_{2n+m,m}$ of $m$-frames in $\mathbb{R}^{2n+m}$ in degrees $* \leq 2n+m-2$.

Now Paechter \cite{Paechter} has computed the homotopy groups of Stiefel manifolds in a range of degrees, which along with the discussion above gives the following.

\begin{lemma}[Paechter]\label{lem:Paechter}
  For $n \geq 2$, we have $\pi_{2n}^s(SO / SO(2n)) \cong \bZ$, and
  the group $\pi_{2n+1}^s(SO/SO(2n))$ is isomorphic to $\bZ/4$ when $n$ is odd
  and to $(\bZ/2)^2$ when $n$ is even.
\end{lemma}

It will not be quite enough for us to know these as abstract groups, we shall need to know a little about their behaviour under the Hurewicz map.

\begin{lemma}\label{lem:HurewiczStuff}
Suppose that $n \geq 2$.
\begin{enumerate}[(i)]
\item\label{it:HurewiczStuff:1} The group $H^{2n+1}(SO/SO(2n);\bF_2)$ is 1-dimensional, generated by a class $x_{2n+1}$ which maps under
$$\partial^* : H^{2n+1}(SO/SO(2n);\bF_2) \lra H^{2n+2}(\mathbf{MO}\langle n \rangle;\bF_2)$$
to $w_{2n+2} \cdot u$, where $u \in H^0(\mathbf{MO}\langle n \rangle;\bF_2)$ is the Thom class. 

\item\label{it:HurewiczStuff:3} The Hurewicz map $\pi_{2n+1}^s(SO/SO(2n)) \to H_{2n+1}(SO/SO(2n);\bF_2)$ is surjective.

\item\label{it:HurewiczStuff:4} The pullback of the Euler class along $SO/SO(2n) \to BSO(2n)$ gives twice a generator of $H^{2n}(SO/SO(2n);\bZ) \cong \bZ$.
\end{enumerate}
\end{lemma}
\begin{proof}
For (\ref{it:HurewiczStuff:1}), we consider the Serre spectral sequence for the fibre sequence
  \begin{equation*}
    SO/SO(2n) \lra BSO(2n) \lra BSO.
  \end{equation*}
  The $\bF_2$-cohomology of $BSO$ is the polynomial algebra on the
  Stiefel--Whitney classes $w_i \in H^i(BSO;\bF_2)$ for $i \geq 2$,
  and the $\bF_2$-cohomology of $BSO(2n)$ is the polynomial algebra on
  the Stiefel--Whitney classes $w_i \in H^i(BSO;\bF_2)$ for $2 \leq i
  \leq 2n$. Thus for each $i \geq 2n$ there must be a class
  $$x_i \in H^i(SO/SO(2n);\bF_2)$$
  which transgresses to $w_{i+1}$, and $H^*(SO/SO(2n);\bF_2)$ is
  isomorphic as a vector space to the exterior algebra on the classes $x_i$. (This could also be computed using the
  Eilenberg--Moore spectral sequence.) As we have assumed that $n \geq 2$,
  it follows that in degrees $2n \leq i \leq 2n+2$ the $i$th
  cohomology of $SO/SO(2n)$ is 1-dimensional and is generated by
  $x_i$. 
  
  The $(3n+1)$-connected map $\Sigma(SO/SO(2n)) \to (BO\langle n
  \rangle) / (BO(2n)\langle n \rangle)$ induces an isomorphism
  $H_{i-1}(SO/SO(2n)) \cong H_i(BO\langle n\rangle, BO(2n)\langle
  n\rangle)$ when $i \leq 3n$, under which the transgression in the
  Serre spectral sequence for the fibre sequence
  \begin{equation*}
    SO/SO(2n) \lra BO(2n)\langle n \rangle \lra BO\langle n \rangle
  \end{equation*}
  may be identified with the connecting homomorphism in the long exact
  sequence for homology of the pair $(BO\langle n\rangle,
  BO(2n)\langle n\rangle)$.
  This long exact sequence is isomorphic, via the Thom isomorphism,
  with the long exact sequence for the cofibre sequence
  $$\mathbf{F}_n \lra \MTtheta_n \lra \Sigma^{-2n} \mathbf{MO}\langle n \rangle.$$
  Hence $x_{2n+1}$ maps to $w_{2n+2} \cdot u$ under the connecting homomorphism.

For (\ref{it:HurewiczStuff:3}), recall that the Hurewicz map for a $(k-1)$-connected space with $k \geq 2$ is an isomorphism in degree $k$ and a surjection in degree $(k+1)$. Hence
\begin{equation}\label{eq:Hurewicz1}
\bZ \cong \pi_{2n}^s(SO/SO(2n)) \lra H_{2n}(SO/SO(2n);\bZ)
\end{equation}
is an isomorphism and
\begin{equation}\label{eq:Hurewicz2}
\pi_{2n+1}^s(SO/SO(2n)) \lra H_{2n+1}(SO/SO(2n);\bZ)
\end{equation}
is a surjection. By \eqref{eq:Hurewicz1} multiplication by 2 on $H_{2n}(SO/SO(2n);\bZ)$ is an injection: it then follows from the Bockstein exact sequence that
$$H_{2n+1}(SO/SO(2n);\bZ) \lra H_{2n+1}(SO/SO(2n);\bF_2)$$
is a surjection, which combined with \eqref{eq:Hurewicz2} gives the result.

For (\ref{it:HurewiczStuff:4}), observe that the map
$$S^{2n} = SO(2n+1)/SO(2n) \lra SO/SO(2n)$$
is $2n$-connected, so induces an injection on $H^{2n}(-;\bZ)$ (and
both spaces have $2n$th cohomology $\bZ$). Now $S^{2n} \to SO/SO(2n)
\to BSO(2n)$ classifies the tangent bundle of $S^{2n}$, which has
Euler number 2, so the pullback of the Euler class to $SO/SO(2n)$ is
not divisible by more than 2; on the other hand, the Euler class
reduces to $w_{2n}$ modulo 2, which vanishes on $SO/SO(2n)$. Hence it
is divisible by precisely 2.
\end{proof}

We now analyse the long exact sequence \eqref{eq:LES} in low degrees.

\begin{lemma}
  The map $\pi_{2n}^s(SO / SO(2n)) \to \pi_0(\MTtheta_n)$ is injective.
\end{lemma}
\begin{proof}
  Under the Thom isomorphism, the Euler class gives a map $E:
  \MTtheta_n \to \mathbf{H}\bZ$, and the composition
  $$\Sigma^{-2n}(SO / SO(2n)) \lra \MTtheta_n \overset{E}\lra \mathbf{H}\bZ$$
  is twice a generator of $H^{2n}(SO/SO(2n);\bZ) \cong \bZ$ by Lemma \ref{lem:HurewiczStuff} (\ref{it:HurewiczStuff:4}). The claim
  follows by taking $\pi_0$ of this composition.
\end{proof}

The long exact sequence \eqref{eq:LES} thus simplifies to
$$\cdots \lra \Omega_{2n+2}^{\langle n \rangle} \overset{\partial_*}\lra \pi_{2n+1}^s(SO / SO(2n)) \lra \pi_1(\MTtheta_n) \overset{s_*}\lra \Omega_{2n+1}^{\langle n \rangle} \lra 0,$$
and we recall that under the isomorphism
$$H_1(\Gamma_{g,1}^n) \overset{\sim}\longleftarrow H_1(B\Diff(W_g^{2n}, D^{2n})) \overset{\sim}\lra H_1(\Omega^\infty_0\MTtheta_n) \cong \pi_1(\MTtheta_n)$$
the map $s_*$ coincides with the map $t$.

\begin{lemma}\label{lem:PartialZero}
  If $n \neq 3$ or $7$ then the map $\partial_* : \Omega_{2n+2}^{\langle n \rangle} \to \pi_{2n+1}^s(SO / SO(2n))$ is zero.
\end{lemma}
\begin{proof}
  By Corollary \ref{cor:Kernel} the kernel of the map $t$, and hence
  $s_*$, has cardinality at least 4, and so the kernel of $s_* :
  \pi_1(\MTtheta_n) \to \Omega_{2n+1}^{\langle n \rangle}$ also has
  cardinality at least 4. On the other hand, the exact sequence and Lemma \ref{lem:Paechter} shows that it has cardinality at most 4, hence it has cardinality precisely 4, so $\partial$ is zero.
\end{proof}

\begin{lemma}\label{lem:PartialOnto}
  If $n = 3$ or $7$ then $\partial_* : \Omega_{2n+2}^{\langle n
    \rangle} \to \pi_{2n+1}^s(SO / SO(2n))$ is surjective (so $s_*$ is
  injective).
\end{lemma}
\begin{proof}
Consider the diagram
  \begin{equation*}
    \xymatrix{
      \pi_{2n+2}(\mathbf{MO}\langle n \rangle) \ar[r]^-{\partial_*} \ar[d]^h& \pi_{2n+1}^s(SO/SO(2n)) \cong \bZ/4 \ar@{->>}[d]^h\\
      H_{2n+2}(\mathbf{MO}\langle n \rangle;\bF_2) \ar[r] & H_{2n+1}(SO/SO(2n);\bF_2) \cong \bZ/2,
    }
  \end{equation*}
  where $h$ denotes the Hurewicz map, and the surjectivity on the
  right is by Lemma \ref{lem:HurewiczStuff} (\ref{it:HurewiczStuff:3}). By Lemma \ref{lem:HurewiczStuff} (\ref{it:HurewiczStuff:1}), the isomorphism  $H_{2n+1}(SO/SO(2n);\bF_2) \cong \bZ/2$ is given by evaluating against the class
  $x_{2n+1} \in H^{2n+1}(SO/SO(2n);\bF_2)$, which under $\partial^*$
  corresponds to $w_{2n+2} \cdot u$. Thus the composition $\partial_* \circ h$
  can be identified with the functional $\Omega_{2n+2}^{\langle n
    \rangle} \to \bZ/2$ given by $[W^{2n+2}] \mapsto \langle [W],
  w_{2n+2}(TW)\rangle$.

  The manifolds $[\bH P^2] \in \Omega_{8}^{\langle 3 \rangle}$ and
  $[\bO P^2] \in \Omega_{16}^{\langle 7 \rangle}$ have Euler
  characteristic 3, so non-trivial top Stiefel--Whitney class. Thus
  $\partial_* \circ h = h \circ \partial_*$ is surjective in these
  cases, but it follows that $\partial_*$ must then be surjective.
\end{proof}

\subsection{Proof of Theorem \ref{thm:Main}}

For the surjectivity part of the statement, we have already explained
how Kreck's result (\cite{KreckAut}) implies the surjectivity of the
homomorphism $f: \Gamma^n_{g,1} \to H_1(\Aut(Q_{W_g^{2n}}))$.  To see
surjectivity of $t \oplus f$ it suffices to see that the restriction
$t\vert_{\Ker(f)}: \Ker(f) \to \Omega^{\langle n\rangle}_{2n+1}$ is
surjective, but that follows from Lemma~\ref{lem:Theta}.

It remains to see that the map $t \oplus f:
H_1(\Gamma_{g,1}^n) \to \Omega^{\langle n\rangle}_{2n+1} \oplus
H_1(\Aut(Q_{W_g^{2n}}))$ is injective for $n\neq 2$ and $g \geq 5$.
For $n=3$ or $7$, the second summand vanishes, so it suffices to prove
that $t$ is injective, which we did in Lemma \ref{lem:PartialOnto}.
In the remaining cases, Lemma~\ref{lem:PartialZero} gives a short
exact sequence fitting into the diagram
\begin{equation*}
\xymatrix{
0 \ar[r]& \pi_{2n+1}^s(SO/SO(2n)) 
\ar[d]
 \ar@{^(->}[r] & \pi_1(\MTtheta_n) \ar@{->>}[r]^-{t} \ar@{=}[d]  &
 \Omega_{2n+1}^{\langle n \rangle} \ar[r]& 0\\
& 
H_1(\mathrm{Aut}(Q_{W_g}))
& H_1(\Gamma_{g,1}^n). \ar@{->>}[l]_-f 
}
\end{equation*}
By the proof of Corollary~\ref{cor:Kernel}, the map
$\pi_{2n+1}^s(SO/SO(2n)) \to H_1(\mathrm{Aut}(Q_{W_g}))$ is an
isomorphism, giving a splitting of the exact sequence in the top row
of the diagram.  This proves Theorem \ref{thm:Main} in these cases.


\section{A filtration of the sphere spectrum}\label{sec:Filtr}

In this section we shall describe and study a filtration of the sphere
spectrum, and a resulting filtration of the stable homotopy groups of
spheres. This plays a role in computing the cobordism groups
$\Omega_{2n+1}^{\langle n \rangle}$ in terms of the stable homotopy
groups of spheres and the $J$-homomorphism.

Recall that $BO\langle n\rangle \to BO$ denotes the $n$-connected
cover, and there is an associated Thom spectrum $\mathbf{MO}\langle n
\rangle$.  Thus $\mathbf{MO}\langle 0\rangle = \mathbf{MO}$, the
spectrum representing unoriented cobordism theory, $\mathbf{MO}\langle
1\rangle = \mathbf{MSO}$, $\mathbf{MO}\langle 2\rangle =
\mathbf{MO}\langle 3\rangle = \mathbf{MSpin}$, etc.  There are maps
\begin{equation*}
  \mathbf{MO} = \mathbf{MO}\langle 0 \rangle \lla \mathbf{MO}\langle 1
  \rangle \lla \mathbf{MO}\langle 2 \rangle \lla \mathbf{MO}\langle 3
  \rangle \lla \cdots 
\end{equation*}
with inverse limit $\mathbf{S}$, the sphere spectrum. We write
$\iota_n : \mathbf{S} \to \mathbf{MO}\langle n \rangle$, and define a
filtration of the stable homotopy groups of spheres by
$$F^n \pi_k(\mathbf{S}) = \mathrm{Ker}\left(\pi_k(\iota_n) : \pi_k(\mathbf{S}) \to \pi_k(\mathbf{MO}\langle n \rangle)\right).$$
Let us write $\overline{\mathbf{MO}\langle n \rangle}$ for the
homotopy cofibre of $\mathbf{S} \to \mathbf{MO}\langle n \rangle$.

\begin{lemma}\label{lem:filt}\mbox{}
  \begin{enumerate}[(i)]
  \item\label{it:1} $F^n \pi_k(\mathbf{S}) =0$ for $k < n$.
  \item\label{it:2} $F^n \pi_k(\mathbf{S})$ contains the image of $J : \pi_k(O) \to \pi_k(\mathbf{S})$ for $k \geq n$.
  \item\label{it:3} $F^n \pi_k(\mathbf{S})$ is equal to the image of $J : \pi_k(O) \to \pi_k(\mathbf{S})$ for $2n \geq k \geq n$.
  \end{enumerate}
\end{lemma}
\begin{proof}
  The spectrum $\overline{\mathbf{MO}\langle n \rangle}$ is
  $n$-connected, and so $\pi_k(\mathbf{S}) \to
  \pi_k(\mathbf{MO}\langle n \rangle)$ is injective for $k < n$; this
  establishes (\ref{it:1}).

  In the cobordism-theoretic interpretation of the homotopy groups of
  spheres, $J(\alpha : S^k \to O)$ is given by the manifold $S^k$ with
  the framing given by twisting the standard (bounding) framing of
  $S^k$ using $\alpha$ to obtain a new framing $\xi_\alpha$.
  \begin{equation*}
    \xymatrix{
      S^k \ar[d] \ar[r]^-{\xi_\alpha} & EO \ar[r] & BO\langle n \rangle \ar[d]\\
      D^{k+1} \ar[rr] & & BO
    }
  \end{equation*}
  While this framing cannot necessarily be extended to $D^{k+1}$, the
  associated $BO\langle n \rangle$-structure can be extended as long
  as $k \geq n$, as the right-hand map is $n$-co-connected; this
  establishes (\ref{it:2}).

  Let $(M^k, \xi)$ be a framed cobordism class representing an element
  of $F^n \pi_k(\mathbf{S})$, so considered as a $BO\langle n
  \rangle$-manifold $M$ bounds a $BO\langle n \rangle$-manifold
  $W$. Now $k+1 \leq 2n+1$ so (similarly to the proof of Lemma \ref{lem:Theta}) by the techniques of \cite[Theorems 5.5
  and 6.6]{KervaireMilnor} we may perform surgery on the interior of $W$ to obtain a new
  $BO\langle n \rangle$-manifold $W'$ which is $\lfloor k/2
  \rfloor$-connected, with the same framed boundary $M$. We may then
  find a handle structure on $W'$ having no handles of index between 1
  and $\lfloor k/2 \rfloor$, and so $W' \setminus D^{k+1}$ is a
  $BO\langle n \rangle$-cobordism from $M$ to $S^k$ which may be
  obtained from $M$ by attaching handles of index at most $k-\lfloor
  k/2 \rfloor \leq n$. As $EO \to BO\langle n \rangle$ is
  $n$-connected, it follows that the framing $\xi$ on $M$ may be
  extended to $W' \setminus D^{k+1}$, and so $(M, \xi)$ is framed
  cobordant to $(S^k, \zeta)$ for some framing $\zeta$ of the
  sphere. But those cobordism classes represented by spheres with some
  framing are precisely the image of the $J$-homomorphism; this
  establishes (\ref{it:3}).
\end{proof}

We wish to understand the group $\Omega_{2n+1}^{\langle n \rangle} = \pi_{2n+1}(\mathbf{MO}\langle n \rangle)$, which is related to $F^n \pi_{2n+1}(\mathbf{S})$ and lies just outside of the range treated in Lemma \ref{lem:filt} (\ref{it:3}). However, the groups $F^n \pi_k(\mathbf{S})$ for $k > 2n$ have been studied, though not quite expressed in this form, by Stolz \cite{Stolz}. Let us explain his technique. 

For $n \geq 2$ the universal (virtual) bundle $\gamma\langle n \rangle$ over
$BO\langle n \rangle$ is Spin, and so the Thom spectrum
$\mathbf{MO}\langle n \rangle$ has a KO-theory Thom class,
$\lambda_n$. There is thus a KO-theory class $\gamma\langle n \rangle \cdot \lambda_n
\in KO^0(\mathbf{MO}\langle n \rangle)$, which we represent by a map
$\alpha_n: \mathbf{MO}\langle n \rangle \to \mathbf{ko}$ to the
connective KO-theory spectrum. As the bundle $\gamma\langle n \rangle \in
KO^0(BO\langle n \rangle)$ becomes trivial when restricted to a point,
the class lifts to $\gamma\langle n \rangle \in KO^0(BO \langle n\rangle,\ast)$, and $\alpha_n$ factors through a map $\alpha'_n :
\overline{\mathbf{MO}\langle n \rangle} \to \mathbf{ko}$, and as
$\overline{\mathbf{MO}\langle n \rangle}$ is $n$-connected this lifts
further to a map
\begin{equation*}
  \overline{\alpha}_n : \overline{\mathbf{MO}\langle n \rangle} \lra
  \mathbf{ko}\langle n \rangle.
\end{equation*}
Stolz defines $A[n+1]$ to be the homotopy fibre of
$\overline{\alpha}_n$. Under the Thom isomorphism we have 
\begin{equation*}
  H^*(\overline{\mathbf{MO}\langle n \rangle}) \cong H^*(BO \langle n
  \rangle, \ast) = H^*(\Omega^\infty(\mathbf{ko}\langle n \rangle), \ast),
\end{equation*}
and using the known cohomology of $BO \langle n \rangle$ and
$\mathbf{ko}\langle n \rangle$ as modules over the Steenrod algebra
Stolz establishes the following.

\begin{theorem}[Stolz \cite{Stolz}]\label{thm:Stolz1}
The spectrum $A[n+1]$ is $(2n+1)$-connected, and
$$\pi_{2n+2}(A[n+1]) = \begin{cases}
\bZ & n+1 \equiv 0, 4 \mod 8\\
\bZ/2 & n+1 \equiv 1, 2 \mod 8\\
0 & \text{otherwise}.
\end{cases}$$
\end{theorem}
Let us write $J : \pi_k(O) \to \pi_k(\mathbf{S})$ for the
$J$-homomorphism. For $\alpha \in \pi_k(O)$, $J(\alpha)$ is given by
the stably framed manifold obtained by changing the bounding framing
on $S^k$ using $\alpha$. As explained in the proof of Lemma
\ref{lem:filt} (\ref{it:2}), the associated $BO\langle n
\rangle$-structure extends canonically over $D^{k+1}$ as long as $k
\geq n$, which gives a map
\begin{equation*}
  \overline{J} : \pi_k(O) \lra \pi_{k+1}(\overline{\mathbf{MO}\langle
    n \rangle})
\end{equation*}
such that $\partial \circ \overline{J} = J$. The composition
$$\pi_k(O) \overset{\overline{J}}\lra \pi_{k+1}(\overline{\mathbf{MO}\langle n \rangle}) \overset{\overline{\alpha_n}}\lra \pi_{k+1}(\mathbf{ko}\langle n \rangle)$$
is an isomorphism (cf.\ \cite[Lemma 3.7]{Stolz}), and it follows from
the commutative diagram
\begin{equation*}
\xymatrix{
 & & \pi_{2n+2}({\mathbf{MO}\langle n \rangle}) \ar[d]\\
0 \ar[r]& \pi_{2n+2}(A[n+1]) \ar[r]& \pi_{2n+2}(\overline{\mathbf{MO}\langle n \rangle}) \ar[r]^-{\overline{\alpha_n}} \ar[d]^-\partial& \pi_{2n+2}(\mathbf{ko}\langle n \rangle) \ar[r]& 0\\
 & & \pi_{2n+1}(\mathbf{S}) \ar[d] & \pi_{2n+1}(O) \ar[l]_-J \ar[lu]_-{\overline{J}} \ar@{=}[u]\\
 & & \pi_{2n+1}({\mathbf{MO}\langle n \rangle}) \ar[d]\\
 & & 0
}
\end{equation*}
that there is an exact sequence
\begin{equation*}\label{eq:exactseq}
\pi_{2n+2}(A[n+1]) \overset{\sigma}\lra \mathrm{Coker}(J)_{2n+1} \lra \pi_{2n+1}({\mathbf{MO}\langle n \rangle}) \lra 0.
\end{equation*}
Hence, given the description of $\pi_{2n+2}(A[n+1])$ in Theorem \ref{thm:Stolz1}, it follows that the quotient $F^n \pi_{2n+1}(\mathbf{S}) / \mathrm{Im}(J)$ is cyclic. Stolz finds various conditions under which the quotient $F^n \pi_{2n+1}(\mathbf{S}) / \mathrm{Im}(J)$ is in fact trivial, i.e.\ the map $\sigma$ is zero.

\begin{theorem}[Stolz \cite{Stolz}]\label{thm:Stolz2}
If either
\begin{enumerate}[(i)]
\item $n+1 \equiv 2 \mod 8$ and $n+1 \geq 18$,

\item $n+1 \equiv 1 \mod 8$ and $n+1 \geq 113$,

\item $n+1 \not\equiv 0, 1,2, 4 \mod 8$,
\end{enumerate}
then $F^n \pi_{2n+1}(\mathbf{S}) = \mathrm{Im}(J)_{2n+1}$. 

If $n+1 = 4\ell$ then $F^n \pi_{2n+1}(\mathbf{S}) / \mathrm{Im}(J)$ is generated by the exotic sphere $\Sigma$ which is the boundary of the manifold obtained by plumbing together two copies of the linear $4\ell$-dimensional disc bundle over $S^{4\ell}$ having trivial Euler class and representing a generator of $\pi_{4\ell}(BO)\cong\bZ$.
\end{theorem}
\begin{proof}
By \cite[Theorem B (i) and (ii)]{Stolz}, in the case $\pi_{2n+2}(A[n+1])=\bZ/2$, these map to zero in $\mathrm{Coker}(J)_{2n+1}$ under the conditions given in the statement of the proposition. 

It follows from \cite[Lemma 10.3]{Stolz} that when $n+1=4\ell$ a generator of $\bZ = \pi_{2n+2}(A[n+1])$ in $\pi_{2n+2}(\overline{\mathbf{MO}\langle n \rangle})$ is given by the class of the plumbing described in the statement of the proposition.
\end{proof}

In the case $n+1 = 4\ell$, it seems to be a difficult problem to obtain any information about the order, or indeed the nontriviality, of $[\Sigma] \in \mathrm{Coker}(J)_{8\ell-1}$. All calculations we have attempted are consistent with the following conjecture.

\vspace{1ex}

\noindent \textbf{Conjecture A.} $[\Sigma]=0 \in \mathrm{Coker}(J)_{8\ell-1}$.

\vspace{1ex}

This conjecture would imply that the map $\sigma$ is zero in these cases too,
and so $\Omega_{8\ell-1}^{\langle 4\ell-1 \rangle} \cong
\mathrm{Coker}(J)_{8\ell-1}$. The most promising approach to this
conjecture seems to be as follows. By the discussion above, the map
$$\mathrm{Coker}(J)_{2n+1} \lra \pi_{2n+1}({\mathbf{MO}\langle n+1 \rangle})$$
is an isomorphism, so Conjecture A is equivalent to

\vspace{1ex}

\noindent \textbf{Conjecture B.} The map $\pi_{8\ell-1}({\mathbf{MO}\langle 4\ell \rangle}) \to \pi_{8\ell-1}({\mathbf{MO}\langle 4\ell-1 \rangle})$ is injective.

\vspace{1ex}

For example, when $\ell=1$ this asks if $\Omega^{\mathrm{String}}_7
\to \Omega^{\Spin}_7$ is injective, which it is as both groups are
zero. When $\ell=2$ this asks if $\mathrm{Coker}(J)_{15} \to
\Omega^{\mathrm{String}}_{15}$ is injective, which it is as
$\pi_{15}({\mathbf{MO}\langle 8 \rangle}) =
\mathrm{Coker}(J)_{15}=\bZ/2$,
$\Omega^{\mathrm{String}}_{15}=\bZ/2$,
and generators of either group may be represented by an exotic sphere
\cite{KervaireMilnor, Giambalvo}.

\begin{corollary}
If $n$ satisfies one of the conditions of Theorem \ref{thm:Stolz2} then the cobordism group $\Omega_{2n+1}^{\langle n \rangle}$ occurring in Theorem \ref{thm:Main} is isomorphic to $\mathrm{Coker}(J)_{2n+1}$.
\end{corollary}


\section{Relation to the work of Kreck}\label{sec:Kreck}

Kreck has given \cite{KreckAut} a description of the mapping class
groups of $(n-1)$-connected $2n$-manifolds, up to two extension
problems. Applied to our situation, he gives extensions
\cite[Proposition 3]{KreckAut}
\begin{equation*}
1 \lra \mathcal{I}_{g,1}^{n} \lra \Gamma^{n}_{g,1} \overset{\hat{f}}\lra \mathrm{Aut}(Q_{W^{2n}_g}) \lra 1
\end{equation*}
and
\begin{equation*}
1 \lra \Theta_{2n+1} \lra \mathcal{I}_{g,1}^{n} \overset{\chi}\lra \mathrm{Hom}(H_n(W_g), S\pi_n(SO(n))) \lra 1
\end{equation*}
where $S\pi_n(SO(n)) = \mathrm{Im}(\pi_n(SO(n)) \to
\pi_n(SO(n+1)))$. These groups are given, for $n \geq 3$, by Table
\ref{table:2} (except that $S\pi_6(SO(6))=0$).
\begin{table}[h]
\centering
\caption{The groups $S\pi_n(SO(n))$, except that $S\pi_6(SO(6))=0$.}
\label{table:2}
\begin{tabular}{c|c c c c c c c c}
$n \mod 8$              & 0 & 1 & 2 & 3 & 4 & 5  & 6 & 7 \\
\hline
$S\pi_n(SO(n))$ & $(\bZ/2)^2$ & $\bZ/2$ & $\bZ/2$ & $\bZ$ & $\bZ/2$ & $0$ & $\bZ/2$ & $\bZ$ 
\end{tabular}
\end{table}

The map $\chi$ may be described as follows: to a diffeomorphism $\varphi : W_{g,1}^{2n} \to W_{g,1}^{2n}$ which acts as the identity on homology, and a class $x \in H_n(W_g^{2n};\bZ) \cong \pi_n(W_{g,1}^{2n})$ represented by an embedding $x : S^n \hookrightarrow W_{g,1}^{2n}$, the sphere $\varphi \circ x$ is isotopic to $x$ and so by the isotopy extension theorem we may suppose that $\varphi \circ x = x$. Then $\epsilon^1 \oplus \nu_x \cong \epsilon^{n+1}$ and the differential $D\varphi\vert_{x(S^n)}$ gives an automorphism of this bundle, corresponding to a map $\chi(\varphi)(x) : S^n \to SO(n+1)$. It can be checked that this map lies in $S\pi_n(SO(n))$.

It is generally difficult to understand the structure
(e.g.\ the non-triviality) of these extensions. To our
knowledge the only case in which complete information is known is
$\Gamma_1^3$, due to Krylov \cite{Krylov}. Crowley \cite{Crowley} has
also been able to solve the extension problem for
$$1 \lra \mathrm{Hom}(H_n(W_g), S\pi_n(SO(n))) \lra \Gamma^{n}_{g,1}/\Theta_{2n+1} \lra \mathrm{Aut}(Q_{W^{2n}_g}) \lra 1$$
for $n=3$ and $7$. An immediate consequence of our Theorem
\ref{thm:Main} is as follows.

\begin{corollary}\label{cor:commutators}
  For $g \geq 5$ the kernel of the composition $\Theta_{2n+1} \to \Gamma_{g,1}^n \to
  \Omega_{2n+1}^{\langle n\rangle}$ is generated by commutators of
  elements of $\Gamma_{g,1}^n$.  In particular, this is true for the
  subgroup $bP_{2n+2} < \Theta_{2n+1} < \Gamma_{g,1}^n$.\qed
\end{corollary}

Our results also be used to shed light on some of these extension problems, especially for those $n$ such that $S\pi_n(SO(n))=0$.

\begin{theorem}
The map
$$t \times \hat{f} : \Gamma^{6}_{g,1} \lra \Omega_{13}^{\langle 6 \rangle} \times \mathrm{O}_{g,g}(\bZ)$$
is an isomorphism, and $\Omega_{13}^{\langle 6 \rangle} = \Omega_{13}^{\mathrm{String}} \cong \bZ/3$.
\end{theorem}
\begin{proof}
We have $S\pi_6(SO(6))=0$, and so the two extensions reduce to
$$1 \lra \Theta_{13} \lra \Gamma^{6}_{g,1} \overset{\hat{f}}\lra \mathrm{O}_{g,g}(\bZ) \lra 1.$$
The group $bP_{14}$ is trivial (\cite{KervaireMilnor}) and so $\Theta_{13} \cong \mathrm{Cok}(J)_{13}$, which is $\bZ/3$, and is isomorphic to $\Omega_{13}^{\langle 6 \rangle}$ by Theorem \ref{thm:RhoIso}. Thus $\Theta_{13} \to \Gamma_{g,1}^{6} \overset{t}\to \Omega_{13}^{\langle 6 \rangle}$ is an isomorphism, which shows that the extension is trivial.
\end{proof}

When $n \equiv 5 \mod 8$, the other case in which
$S\pi_n(SO(n)) = 0$, we also solve the extension problem left open by
Kreck.

\begin{theorem}\label{thm:Extension}
If $n \equiv 5 \mod 8$ then there is a central extension
$$1 \lra bP_{2n+2} \lra  \Gamma^{n}_{g,1} \overset{t \times \hat{f}}\lra \Omega_{2n+1}^{\langle n \rangle} \times \mathrm{Sp}^q_{2g}(\bZ) \lra 1,$$
where we write $\mathrm{Sp}^q_{2g}(\bZ) \leq \mathrm{Sp}_{2g}(\bZ)$
for the subgroup of those automorphisms of the symplectic space
$\bZ^{2g}$ which preserve the standard quadratic function. We have $\Omega_{2n+1}^{\langle n \rangle} \cong
\mathrm{Cok}(J)_{2n+1}$, and if $g \geq 5$ then the
subgroup $bP_{2n+2} \leq \Gamma^{n}_{g,1}$ is generated by
commutators.
\end{theorem}
\begin{proof}
We have $S\pi_n(SO(n))=0$ and so Kreck's exact sequences reduce to
$$1 \lra \Theta_{2n+1} \lra  \Gamma^{n}_{g,1} \overset{\hat{f}}\lra \mathrm{Sp}^q_{2g}(\bZ) \lra 1.$$
Furthermore, in this dimension the Kervaire--Milnor \cite{KervaireMilnor} exact sequence is
\begin{equation}\label{eq:Theta}
1 \lra bP_{2n+2} \lra \Theta_{2n+1} \lra \mathrm{Coker}(J)_{2n+1} \lra 1
\end{equation}
and $n+1 \equiv 6 \mod 8$ so by Theorem \ref{thm:RhoIso} the map $\mathrm{Cok}(J)_{2n+1} \to \Omega_{2n+1}^{\langle n \rangle}$ is an isomorphism. Thus the kernel of $t \times \hat{f}$ is precisely the subgroup $bP_{2n+2} \leq \Theta_{2n+1} \leq \Gamma_{g,1}^{n}$. Furthermore, we know $t \times \hat{f}$ induces an isomorphism on abelianisations for $g \geq 5$, so $bP_{2n+2}$ consists of commutators.
\end{proof}

Finally, we determine the extension in Theorem \ref{thm:Extension}.
It can be pulled back to a central extension
\begin{equation}\label{eq:Egn}
1 \lra bP_{2n+2} \lra  E(g, n) \lra \mathrm{Sp}^q_{2g}(\bZ) \lra 1,
\end{equation}
with $E(g,n) = \mathrm{Ker}(t: \Gamma_{g,1}^n \to
\Omega_{2n+1}^{\langle n\rangle})$.  Brumfiel \cite{BrumfielI,
  BrumfielII, BrumfielIII} has constructed a splitting
of~\eqref{eq:Theta} (at least for for $n \neq 2^k-2$, which is
satisfied as we are supposing that $n \equiv 5 \mod  8$).
Any splitting $s : \mathrm{Cok}(J)_{2n+1} \to \Theta_{2n+1}$ gives
rise to a composition
\begin{equation*}
  \mathrm{Cok}(J)_{2n+1} \overset{s}\lra \Theta_{2n+1} \lra
  \Gamma_{g,1}^{n} \overset{t}\lra \Omega_{2n+1}^{\langle n \rangle}
\end{equation*}
which is an isomorphism.  As $\Theta_{2n+1}$ lies in the centre of
$\Gamma_{g,1}^{2n}$, we obtain a splitting
\begin{equation}\label{eq:2}
  \Gamma_{g,1}^{n} \cong E(g,n) \times \mathrm{Cok}(J)_{2n+1},
\end{equation}
giving the following improvement to Corollary~\ref{cor:commutators}: For $g \geq 5$
the subgroup $bP_{2n+2} < \Theta_{2n+1} < \Gamma_{g,1}^n$ is generated
by commutators of elements from the subgroup $E(g,n) <
\Gamma_{g,1}^n$.  Hence $bP_{2n+2}$ vanishes in the abelianisation of
$E(g,n)$,
and in fact we may deduce
that the group homomorphism $E(g,n) \to \mathrm{Sp}^q_{2g}(\bZ)$
induces an isomorphism of abelianisations.  We shall use this fact to
determine the class of the extension \eqref{eq:Egn} in Theorem~\ref{thm:extension} below.
\begin{lemma}\label{lem:split}
  The homomorphism $\bZ/4\bZ \to \mathrm{Sp}^q_{2g}(\bZ)$ which sends
  the generator to the matrix
  \begin{equation*}
    X_g = \mathrm{diag}\bigg(
    \begin{pmatrix}
      0 & -1\\ 1 & 0
    \end{pmatrix},
    \begin{pmatrix}
      1 & 0\\ 0 & 1
    \end{pmatrix},
    \dots,
    \begin{pmatrix}
      1 & 0\\ 0 & 1
    \end{pmatrix}
    \bigg),
  \end{equation*}
  admits a lift to $E(g,n)$.

  For $g \geq 5$ the resulting homomorphisms $\bZ/4\bZ \to E(g,n) \to
  \mathrm{Sp}_{2g}^q(\bZ)$ both induce isomorphisms in $H_1(-;\bZ)$
  and hence in the torsion subgroups of $H^2(-;\bZ)$.
\end{lemma}
\begin{proof}
  Using the standard embedding $W_{1,0}^{2n} = S^n \times S^n \subset
  \R^{n+1} \times \R^{n+1}$ it is easy to lift the matrix $X_1$ to a
  diffeomorphism of $W_{1,0}^{2n}$, namely the restriction of the
  linear map $(x_1, \dots, x_{n+1}, y_1, \dots, y_{n+1}) \mapsto
  (-y_1, y_2, \dots, y_{n+1}, x_1, \dots, x_{n+1})$.  We obtain an
  order-four element $z' \in \Gamma_{1,0}^n$, which by Lemma
  \ref{lem:Closed} lifts to an order-four element $z'' \in
  \Gamma_{1,1}^n$ with $\hat f(z'') = X_1 \in \mathrm{Sp}_{2}^q(\bZ)$. This may be stabilised to an order-four element $z''_g \in
  \Gamma_{g,1}^n$ with $\hat f(z''_g) = X_g \in
  \mathrm{Sp}_{2g}^q(\bZ)$.  The element $z''_g$ may not lie in the
  subgroup $E(g,n) = \mathrm{Ker}(t)$, but we may use the
  splitting~\eqref{eq:2} to project it to an element $z_g \in E(g,n)$
  with $z_g^4 = 1$.  Since $\hat f(z_g) = X_g \in
  \mathrm{Sp}_{2g}^q(\bZ)$ this gives the required lift.

  For the claim about $H_1(-;\bZ)$, we have already seen that $E(g,n)
  \to \mathrm{Sp}_{2g}^q(\bZ)$ induces an isomorphism of
  abelianisations for $g \geq 5$.  For $\bZ/4\bZ$, it follows from the
  formula in \cite[p.\ 147]{JohnsonMillson} that the composition
  $\bZ/4\bZ \to \mathrm{Sp}_{2g}^q(\bZ) \to
  H_1(\mathrm{Sp}_{2g}^q(\bZ);\bZ)$ is an isomorphism (this only
  requires $g \geq 3$).
\end{proof}

By the second part of this lemma, the maps $\bZ/4 \to E(g,n) \to
\mathrm{Sp}_{2g}^q(\bZ)$ induce isomorphisms on torsion subgroups of
$H^2(-;\bZ)$, and hence give compatible splittings
\begin{align}\label{eq:splittings}
\begin{split}
H^2(\mathrm{Sp}_{2g}^q(\bZ);\bZ) &\overset{\sim}\lra H^2(\bZ/4;\bZ) \oplus \Hom(H_2(\mathrm{Sp}_{2g}^q(\bZ);\bZ), \bZ)\\
H^2(E(g,n);\bZ) &\overset{\sim}\lra H^2(\bZ/4;\bZ) \oplus \Hom(H_2(E(g,n);\bZ), \bZ)
\end{split}
\end{align}
of the universal coefficient sequences.

\begin{lemma}\label{lem:Mu}
  For $g \geq 5$ there is a unique class $\mu \in H^2(\mathrm{Sp}^q_{2g}(\bZ);\bZ)$
  with the following properties.
  \begin{enumerate}[(i)]
  \item\label{item:4} For any closed oriented surface $S$ and map $f:
    S \to B\mathrm{Sp}^q_{2g}(\bZ)$, the signature of the symmetric form
    \begin{equation*}
      \langle -, - \rangle: H^1(S;\bZ^{2g}_f) \otimes H^1(S;\bZ^{2g}_f) \xrightarrow{\cup}
      H^2(S; \bZ^{2g}_f \otimes \bZ^{2g}_f) \xrightarrow{\omega}
      H^2(S;\bZ) \xrightarrow{[S]} \bZ
    \end{equation*}
    agrees with $8 (f^* \mu)[S]$.
  \item\label{item:5} $\mu$ is in the kernel of the map
    $H^2(\mathrm{Sp}^q_{2g}(\bZ);\bZ) \to H^2(\bZ/4\bZ;\bZ)$,
    where $\bZ/4\bZ \to \mathrm{Sp}^q_{2g}(\bZ)$ is the homomorphism from
    Lemma~\ref{lem:split}.
  \end{enumerate}
\end{lemma}
\begin{proof}
We first claim that the indicated symmetric form $\langle -, -
\rangle$ is even, and hence has signature divisible by 8. To prove
this, we may reduce the form modulo 2 and show that $\langle x,x
\rangle \equiv 0  \mod  2$ for any $x \in H^1(S;\bZ^{2g}_f)$. In order
to compute this we choose a triangulation $S \approx \vert K \vert$
for an ordered simplicial complex $K$, and let $\varphi \in C^1(K;\bZ_f^{2g})$ be a simplicial cocycle, which assigns to each 1-simplex $[v_0,v_1] \in K$ a section $\varphi_{[v_0,v_1]}$ of the coefficient system $\bZ_f^{2g}$ over $[v_0,v_1]$. Then by the Alexander--Whitney formula we have
$$\omega(\varphi \cup \varphi)([v_0, v_1, v_2]) = \omega_{v_1}(\varphi_{[v_0,v_1]}, \varphi_{[v_1, v_2]}),$$
where $\omega_{v_1}(-,-)$ is the bilinear form on $\bZ_f^{2g}$ over the point $v_1$, and so
$$\langle \varphi, \varphi \rangle \equiv \sum_{\mathclap{[v_0, v_1, v_2] \in K}} \omega_{v_1}(\varphi_{[v_0,v_1]}, \varphi_{[v_1, v_2]}) \mod 2,$$
where the sum is taken over all 2-simplices of $K$ (note that the choice of ordering of the vertices of the 2-simplex does not affect this formula, as we are working modulo 2 and $\omega_{v_1}(a,a)=0$ by skew-symmetry). As $\varphi$ is a cocycle we have $\varphi_{[v_0,v_1]} + \varphi_{[v_1, v_2]} = \varphi_{[v_0, v_2]}$, and so using the quadratic refinement $q_{v_1}(-)$ associated to the bilinear form $\omega_{v_1}(-,-)$ reduced modulo 2 we obtain
$$q_{v_1}(\varphi_{[v_0, v_2]}) = q_{v_1}(\varphi_{[v_0,v_1]}) + q_{v_1}(\varphi_{[v_1, v_2]}) + \omega_{v_1}(\varphi_{[v_0,v_1]}, \varphi_{[v_1, v_2]}) \mod 2.$$
Hence
$$\langle \varphi, \varphi \rangle \equiv \sum_{\mathclap{[v_0, v_1, v_2] \in K}} q_{v_1}(\varphi_{[v_0, v_2]}) + q_{v_1}(\varphi_{[v_0,v_1]}) + q_{v_1}(\varphi_{[v_1, v_2]}) \mod 2,$$
but as each 1-simplex is the face of precisely two 2-simplices, this sum is zero. Hence the form $\langle -, - \rangle$ is even as claimed, so has signature divisible by 8.

Now note that the signature of this form only depends on the oriented
cobordism class of $f: S \to B\mathrm{Sp}_{2g}^q(\bZ)$, or in other
words on the homology class $f_*([S])$. Hence it can be written as
$8s(f_*[S])$ for a (unique) homomorphism
$$s = \mathrm{sign}/8 : H_2(\mathrm{Sp}_{2g}^q(\bZ);\bZ) \lra \bZ.$$
By the universal coefficient theorem, this proves the existence of a $\mu$
satisfying~(\ref{item:4}) and determines it uniquely up to adding any
torsion element. The splitting \eqref{eq:splittings} of $H^2(\mathrm{Sp}_{2g}^q(\bZ);\bZ)$ and~(\ref{item:5}) uniquely determines the torsion summand.
\end{proof}

\begin{definition}
  Let $\bZ \to E_g \to \mathrm{Sp}^q_{2g}(\bZ)$ be the central extension
  classified by the class $\mu$.  For $d \in \bZ$, let $\bZ/d\bZ \to
  E_{g,d} \to \mathrm{Sp}^q_{2g}(\bZ)$ be the extension with $E_{g,d} =
  E_g/d\bZ$, i.e.\ the extension classified by the image of $\mu$ in
  cohomology modulo $d$.
\end{definition}

\begin{theorem}\label{thm:extension}
  Let $n \equiv 5 \mod 8$ and $g \geq 5$.  The homomorphism
  $E(g,n) \to \mathrm{Sp}_{2g}^q(\bZ)$ obtained by restricting $\hat
  f$ lifts to an isomorphism
  \begin{equation}\label{eq:square}
\begin{gathered}
    \xymatrix{
      E(g,n) \ar[r]^-{\cong}\ar[d] & E_{g,|bP_{2n+2}|}\ar[d]\\
      {\mathrm{Sp}_{2g}^q(\bZ)} \ar@{=}[r] & 
      {\mathrm{Sp}_{2g}^q(\bZ)}.
    }
\end{gathered}
  \end{equation}
\end{theorem}

\begin{proof}
  To produce a lift of $\hat f$, it suffices to prove $\mu$ becomes
  divisible by $|bP_{2n+2}|$ when pulled back to $H^2(E(g,n); \bZ)$.
  By the splittings \eqref{eq:splittings} and the characterisation of
  $\mu$ in Lemma~\ref{lem:Mu}, in order to do this it is enough to
  show that the map
  \begin{equation*}
    H_2(E(g,n);\bZ) \overset{i_*}\lra H_2(\Gamma_{g,1};\bZ)
    \overset{\hat f_*}\lra H_2(\mathrm{Sp}_{2g}^q(\bZ);\bZ) \overset{\mathrm{sign}}\lra \bZ
  \end{equation*}
  is divisible by $8\cdot |bP_{2n+2}|$, where $i: E(g,n) \to
  \Gamma_{g,1}^n$ is the inclusion and the map $\mathrm{sign}$ sends a second homology
  class to the signature of the symmetric form described in Lemma
  \ref{lem:Mu} (\ref{item:4}).

  Firstly, by the splitting \eqref{eq:2} and the K{\"u}nneth theorem,
  the map $i_*$ is an isomorphism modulo torsion, so the divisibility
  of the map $\mathrm{sign} \circ \hat f_* \circ i_*$ is the same as that of $\mathrm{sign}
  \circ \hat f_*$. Secondly, the composition
  $$H_2(B\Diff_\partial(W_{g,1});\bZ) \lra H_2(\Gamma_{g,1};\bZ)
  \overset{\hat f_*}\lra H_2(\mathrm{Sp}_{2g}^q(\bZ);\bZ)
  \overset{\mathrm{sign}}\lra \bZ$$ has the first map surjective and sends a
  smooth bundle $E \to S$ with fibres $W_{g,1}$ and base a closed
  oriented surface to the signature of $H^1(S;H^n(W_{g,1};\bZ))$ and
  so by \cite{CHS} to the signature of the total space $E$ (with $S
  \times D^{2n}$ glued in to make it a closed manifold). This total
  space defines a class in $\Omega_{2n+2}^{\langle n\rangle} =
  \Omega_{2n+2}^{\langle n+1\rangle}$, and is thus cobordant to a closed
  smooth manifold which is framed away from a point. By \cite[p.\
  457]{KM2} the signature of such a manifold is divisible by $8\cdot
  |bP_{2n+2}|$, as required.

  We have constructed the map $E(g,n) \to E_{g,|bP_{2n+2}|}$ making
  the square \eqref{eq:square} commute, and it remains to prove that the induced map of
  kernels $bP_{2n+2} \to \bZ/|bP_{2n+2}|$ is an isomorphism.  To see
  this, we consider the induced map of Serre spectral sequences, and in particular the commutative square
\begin{equation*}
\xymatrix{
H_2(\mathrm{Sp}_{2g}^q(\bZ);\bZ) = E^2_{2,0} \ar[r]^-{d^2} \ar@{=}[d]& E^2_{0,1} = H_1(bP_{2n+2};\bZ) = bP_{2n+2}\ar[d]\\
H_2(\mathrm{Sp}_{2g}^q(\bZ);\bZ) = E^2_{2,0} \ar[r]^-{d^2}& E^2_{0,1} =H_1(\bZ/\vert bP_{2n+2} \vert;\bZ) = \bZ/\vert bP_{2n+2} \vert.
}
\end{equation*}
The lower horizontal map is identified with
$$H_2(\mathrm{Sp}_{2g}^q(\bZ);\bZ) \overset{s}\lra \bZ \lra \bZ/\vert bP_{2n+2} \vert$$
so is surjective if $s = \mathrm{sign}/8$ is indivisible. In this case the right-hand
vertical map must also be surjective, and hence be an isomorphism, as
both groups are cyclic of the same order.

To show $s$ is indivisible, consider the maps
$$\mathrm{Sp}_{2g}(\bZ, 2) \lra \mathrm{Sp}_{2g}^q(\bZ) \lra \mathrm{Sp}_{2g}(\bZ)$$
from the level 2 congruence subgroup and to the full symplectic
group. Meyer has shown that the signature map $\mathrm{sign} : H_2(\mathrm{Sp}_{2g}(\bZ);\bZ) = \bZ \to \bZ$, defined as in Lemma~\ref{lem:Mu} (\ref{item:4}), has image $4\bZ$ as
long as $g \geq 3$ \cite[Satz 2]{Meyer}. Putman has shown that
$H_2(\mathrm{Sp}_{2g}(\bZ, 2);\bZ) \to H_2(\mathrm{Sp}_{2g}(\bZ);\bZ)$
has image $2\bZ$ \cite[Theorem F]{PutmanLevel} as long as $g \geq
4$. Thus the signature map restricted to the level 2 congruence
subgroup has image $8\bZ$, so in particular it hits $8 \in \bZ$ for $g
\geq 4$, and so the signature map restricted to
$\mathrm{Sp}_{2g}^q(\bZ)$ does too; hence $s$ hits $1 \in \bZ$.
\end{proof}


\bibliographystyle{amsalpha}
\bibliography{biblio}

\end{document}